\documentclass{amsart}

\usepackage{amsmath,amsthm,amssymb}
\usepackage{graphicx}
\usepackage[all]{xy}

\newtheorem{theorem}{Theorem}[section]
\newtheorem{lemma}[theorem]{Lemma}
\newtheorem{proposition}[theorem]{Proposition}
\newtheorem{claim}[theorem]{Claim}
\newtheorem{criterion}[theorem]{Criterion}

\theoremstyle{definition}

\theoremstyle{remark}
\newtheorem{remark}{Remark}[section]

\newcommand{\alt}{\mathop{\mathrm{alt}}\nolimits}
\newcommand{\fix}{\mathop{\mathrm{Fix}}\nolimits}
\newcommand{\id}{\mathop{\mathrm{id}}\nolimits}

\title[Seifert surgery and Rasmussen invariant]
{Seifert fibered surgery and Rasmussen invariant}

\dedicatory{In honour of J.~Hyam Rubinstein}

\date{\today}

\author[K.~Ichihara]{Kazuhiro Ichihara}
\address{Department of Mathematics, College of Humanities and Sciences, Nihon University,
3-25-40 Sakurajosui, Setagaya-ku, Tokyo 156-8550, Japan}
\email{ichihara@math.chs.nihon-u.ac.jp}
\thanks{The first author is partially supported by
Grant-in-Aid for Young Scientists (B), No.~20740039,
Ministry of Education, Culture, Sports, Science and Technology, Japan.}

\author[I.~D.~Jong]{In Dae Jong}
\address{Faculty of Liberal Arts and Sciences, 
Osaka Prefecture University, 1-1 Gakuen-cho, Nakaku, Sakai, Osaka 599-8531, Japan}
\email{jong@las.osakafu-u.ac.jp}
\thanks{The second author is partially supported by 
Grant-in-Aid for Research Activity Start-up, No.~22840037, 
Japan Society for the Promotion of Science.}

\keywords{Alternation number; Montesinos trick; Pretzel knot; Rasmussen invariant; Seifert fibered surgery; Signature.}

\subjclass[2000]{Primary 57M50; Secondary 57M25}

\begin{document}

\begin{abstract}
We give a new criterion for a given knot to be a Montesinos knot 
by using the Rasmussen invariant and the signature. 
We apply the criterion to study Seifert fibered surgery on a strongly invertible knot, 
and show that a $(p,q,q)$-pretzel knot with integers $p,q \ge 2$ admits no Seifert fibered surgery. 
\end{abstract}

\maketitle

\section{Introduction}
The well-known Hyperbolic Dehn Surgery Theorem 
\cite[Theorem 5.8.2]{Thurston1978} 
says that all but finitely many Dehn surgeries on a hyperbolic knot 
yield hyperbolic manifolds. 
Here a knot is called \textit{hyperbolic} 
if its complement admits a complete hyperbolic structure of finite volume. 
Thereby a Dehn surgery on a hyperbolic knot yielding a non-hyperbolic manifold is called {\it exceptional}. 
In view of this, it is an interesting and challenging problem to determine and classify 
all exceptional surgeries on hyperbolic knots in the $3$-sphere $S^3$. 

Note that 
exceptional surgeries can be classified into the following three types; 
a Seifert fibered surgery, a toroidal surgery, a reducible surgery. 
Here a Dehn surgery is called {\it Seifert fibered} / {\it toroidal} / {\it reducible} 
if it yields a Seifert fibered / toroidal / reducible manifold respectively. 
This classification has been achieved as 
a consequence of the affirmative solution 
to the Geometrization Conjecture. 
Here we note that currently less is known about 
a Seifert fibered surgery compared with reducible or toroidal surgeries. 
See \cite{Boyer2002} for a survey. 

In this paper, we give a sufficient condition for a strongly invertible knot in $S^3$ 
to admit no Seifert fibered surgery (Proposition~\ref{cri:SFSquotient}). 
To use this sufficient condition, we need to detect whether a given knot is a Montesinos knot or not. 
For this problem, we give a new criterion for 
a knot to be a Montesinos knot (Criterion~\ref{cri:Monte}), 
which is based on estimations of two concordance invariants, 
that is, the signature and the Rasmussen invariant. 
This criterion is obtained from the fact that any Montesinos knot is close to some alternating knot in the sense of the Gordian distance. 
This would be interesting independently.

As an application, we study Seifert fibered surgeries on pretzel knots. 
A \textit{pretzel knot} of type $(a_1, a_2, \dots, a_l)$, 
denoted by $P(a_1, a_2, \dots, a_l)$, is defined as 
a knot admitting a diagram obtained by putting rational tangles 
of the form $1/a_1, 1/a_2, \dots, 1/a_l$ together in a circle. 
For basic terminologies in knot theory, 
we refer the reader to \cite{KawauchiBook}, \cite{RolfsenBook}. 

Among all knots in $S^3$, 
exceptional surgeries on pretzel knots have been studied extensively. 
(We summarize such results later.) 
One of the motivations to study Dehn surgeries on pretzel knots 
is that many interesting examples about exceptional surgeries 
have been found among pretzel knots. 
For example, the first known examples of 
Dehn surgeries on hyperbolic knots yielding lens spaces 
are those on $P(-2,3,7)$~\cite{FintushelStern1980}, and 
the first example of a Seifert fibered surgery 
on a non-strongly invertible knot is that on $P(-3,3,5)$~\cite{MattmanMiyazakiMotegi2006}. 

On the other hand, we have the following. 

\begin{theorem}\label{thm:inv}
A pretzel knot $P(p,q,q)$ with integers $p , q \ge 2$ 
admits no Seifert fibered surgery. 
\end{theorem}

Note that the integer $q$ must be odd otherwise $P(p,q,q)$ is not a knot but a link. 

\begin{remark}\label{rem:hyp-pretzel}
A pretzel knot $P(p,q,q)$ with integers $p , q \ge 2$ 
is known to be hyperbolic. 
Actually it is already known which pretzel knots are non-hyperbolic. 
Let $T_{a,b}$ denote the torus knot of type $(a,b)$, that is, the knot isotopic to the $(a,b)$-curve on the standardly embedded torus in $S^3$. 
If the length $l$ of a pretzel knot is less than three, 
then it actually is a trivial knot or $T_{2,x}$, which is non-hyperbolic. 
Otherwise, non-trivial non-hyperbolic pretzel knots are 
just $P(-2,3,3)$ and $P(-2,3,5)$, which are actually $T_{3,4}$ and $T_{3,5}$ respectively. 
This fact was shown by Kawauchi~\cite{Kawauchi1985}. 
Oertel~\cite[Corollary 5]{Oertel1984} independently showed the fact together with the result in the unpublished monograph by Bonahon and Siebenmann~\cite{BonahonSiebenmann1979-85, BonahonSiebenmann2010}. 
\end{remark}

In the following, as a background, we collect known facts on exceptional surgeries on pretzel knots. 
Most of the following results concern a large class of knots, called Montesinos knots. 
However, for simplicity, we only deal with pretzel knots. 
See the original references for the precise statements. 

If the length $l$ of a pretzel knot is less than three, 
then it is non-hyperbolic as noted in Remark~\ref{rem:hyp-pretzel}. 
On the other hand, it was shown by Wu~\cite{Wu1996} that 
a pretzel knot $P(a_1,\dots, a_l)$ with $|a_i| \ge 2$ and $l \ge 4$ admits no exceptional surgery. 
Also Wu showed that a hyperbolic pretzel knot admits no reducible surgery~\cite{Wu1996}, 
and gave a complete classification of toroidal surgeries on pretzel knots~\cite{Wu2006}. 
The authors also showed that there is no toroidal Seifert fibered surgery on pretzel knots other than the trefoil knot~\cite{IchiharaJong2010a}. 
Among atoroidal Seifert fibered surgeries on pretzel knots, 
Dehn surgeries yielding $3$-manifolds with cyclic or finite fundamental groups 
were completely classified by the authors~\cite{IchiharaJong2009} and 
Futer, Ishikawa, Kabaya, Mattman, and Shimokawa~\cite{FuterIshikawaKabayaMattmanShimokawa2009} independently. 

Recently, Wu~\cite{Wu2009, Wu2010} studied atoroidal Seifert fibered surgeries in detail. 
In particular, he showed that if a hyperbolic pretzel knot that is not equivalent to a two-bridge knot admits an atoroidal Seifert fibered surgery, 
then it is equivalent to $P(q_1,q_2,q_3,n)$ with $n=0,-1$ and, 
up to relabeling, $(|q_1|,|q_2|,|q_3|) = (2,|q_2|,|q_3|), (3,3,|q_3|), \mbox{ or } (3,4,5)$~\cite[Theorem 7.2]{Wu2010}. 
Our theorem given in this paper shows that, among the families left open by Wu, 
$P(2,q,q)$ and $P(3,3,n)$ with $q \ge 3, n \ge 3$ have no Seifert fibered surgeries. 
This is one of the motivations for focusing on the family $P(p,q,q)$ with $p,q \ge 2$. 
Furthermore the authors and Kabaya~\cite{IchiharaJongKabaya2011} gave 
a complete classification of exceptional surgeries on $(-2,q,q)$-pretzel knots with $q \ge 3$, 
in particular, it was shown that a $(-2,p,p)$-pretzel knot with $q \ge 3$ admits no Seifert fibered surgery. 

Consequently we see that 
if a hyperbolic pretzel knot admits an atoroidal Seifert fibered surgery, then it is equivalent to 
\begin{itemize}
\item[(i)] $P( \pm 2 , a , b )$ with $a \ne b$, 
\item[(ii)] $P( 3 , 3 , c)$ with $c \le -3$, 
\item[(iii)] $P( 3 , 3 , d, -1)$ with $d \ge 3 $, 
\item[(iv)] $P( 3 , - 3 , \pm e)$ with $e \ge 4$, or 
\item[(v)] $P( 3 , \pm 4 , \pm 5)$, $P( 3 , \pm 4 , \mp 5)$,  $P( 3 , 4 , 5, -1)$. 
\end{itemize}

We here include some comments for the next step of our studies. 
Our methods presented in this paper may be applied to knots different from $P(p,q,q)$. 
One of the candidates would be a kind of double-torus knots. 
A key point to apply our techniques is that 
the knot can admit a particular kind of symmetry of period 2. 
(A typical example of such a symmetry is depicted in Figure~\ref{fig:axis1} for $P(p,q,q)$.) 
Unfortunately it seems that the pretzel knots listed above do not have such a symmetry, 
and so, there is little chance that our techniques would work for the knots.

\section{Criteria}\label{sec:cri}

In this section, 
we introduce a sufficient condition for a strongly invertible knot to admit no Seifert fibered surgery. 
We also give a criterion for a knot to be a Montesinos knot. 

First we set up definitions and notations. 
Let $K$ be a knot in $S^3$, and let $r$ be a rational number. 
Suppose that $r$ corresponds to a slope $\gamma$ 
(i.e., an isotopy class of non-trivial simple closed curves) 
by the well-known correspondence between 
$\mathbb{Q} \cup \{ 1/0 \}$ and 
the slopes on the boundary torus $\partial E(K)$, 
which is given by using the standard meridian-longitude system for $K$.
Then, the \textit{Dehn surgery on $K$ along $r$} 
is defined as the following operation. 
Take the exterior $E(K)$ of the knot $K$ (i.e., remove an open tubular neighborhood of $K$), 
and glue a solid torus $V$ to $E(K)$ 
so that a simple closed curve representing $\gamma$ bounds a meridian disk in $V$. 
We call such a Dehn surgery 
the \textit{$r$-surgery} on $K$ for brevity, 
and denote the obtained manifold by $K(r)$.

\subsection{Criterion for a strongly invertible knot to admit no Seifert fibered surgery}\label{subsec:cri-SeifSurg}

Let $K$ be a strongly invertible knot, and let $r$ be a rational number. 
Here a knot $K$ is said to be {\it strongly invertible} if there is an orientation preserving involution of $S^3$ which induces an involution of $K$ with two fixed points. 
Applying the Montesinos trick~\cite{Montesinos1975}, we obtain a link $L_r$ in $S^3$ such that the double branched covering space of $S^3$ branched along $L_r$, denoted by $M_2(L_r)$, is homeomorphic to $K(r)$. 
Actually, $L_r$ is either a knot or a two-component link. 
Suppose that $K(r) \cong M_2(L_r)$ is a Seifert fibered manifold with the base orbifold $S^2$. 
If $K$ is non-trivial and not equivalent to the trefoil knot, then one can take the covering involution $\iota : M_2(L_r) \rightarrow M_2(L_r)$ with $M_2(L_r) / \iota \cong S^3$ as fiber preserving (see \cite[Lemma 3.1]{Motegi2003}). 
Let $\pi : K(r) \rightarrow S^2$ be a Seifert fibration, $\iota$ the involution preserving the Seifert fibration, and $\hat{\iota}: S^2 \rightarrow S^2$ the homeomorphism induced from $\iota$ satisfying $\pi \circ \iota = \hat{\iota} \circ \pi$. 
Let $\fix(\iota)$ denote the set of fixed points of $\iota$. 
If each component of $\fix(\iota)$ is a fiber in $K(r)$, then $K(r) \setminus \fix(\iota)$ admits a Seifert fibration. 
Since $\iota$ preserves the Seifert fibration of $K(r)$, the exterior of $L_r$ also admits a Seifert fibration, namely, $L_r$ is a {\it Seifert link}. 
If a component of $\fix(\iota)$ is not a fiber of $K(r)$, then $\hat{\iota}$ is orientation reversing (see \cite[Lemma 3.2 (1)]{Motegi2003}) and $\iota$ reverses an orientation of each fiber in $K(r)$.  
In this case, every cone point is lying on the circle $C_{\iota}=\pi(\fix(\iota))$~\cite[Lemma 7.2]{Motegi2003} and $\hat{\iota}$ is a reflection on $S^2$ along $C_{\iota}$. 
Thus, $L_r$ is equivalent to a Montesinos link. 
For more details, we refer the reader to \cite{MiyazakiMotegi2002a} and \cite{Motegi2003}. 
Consequently, we obtain the following.

\begin{proposition}\label{cri:SFSquotient}
Let $K$ be a strongly invertible hyperbolic knot and let $r$ be a rational number. 
Let $L_r$ be a link obtained by applying the Montesinos trick to $K(r)$. 
If $L_r$ is equivalent to neither a Seifert link nor a Montesinos link, 
then $K(r)$ is not a Seifert fibered manifold with the base orbifold $S^2$. 
\end{proposition}

To use Proposition~\ref{cri:SFSquotient} we need to show that the link $L_r$ is neither a Seifert link nor a Montesinos link. 
We consider this problem in the following.

\subsection{Seifert link}\label{subsec:S-link}

Seifert links are well-understood and completely classified. 
Here we review a classification of Seifert links consisting of at most two components. 

\begin{lemma}[\cite{BurdeMurasugi1970}, see also {\cite[Proposition 7.3]{EisenbudNeumann1985}}]\label{lem:SeifertLink}
Let $L$ be a Seifert link in $S^3$ which consists of at most two components. 
Then $L$ is equivalent to one of the following: 
\begin{enumerate}
\item[(S1)] A torus knot. 
\item[(S2)] A two-component torus link. 
\item[(S3)] A two-component link which consists of a torus knot and a core curve of the torus.
\end{enumerate}
\end{lemma}

Here we give a lemma to detect whether a given knot is a torus knot, which will be used in Section~\ref{sec:proof}. 
For a knot $K$, we denote by $\det(K)$ the {\it determinant} of $K$ and by $\Delta_K(t)$ the Alexander polynomial of $K$. 
Note that $\det(K) = |\Delta_K(-1)|$. 
Then we have the following. 

\begin{lemma}\label{lem:detT4}
Let $x$ be a positive odd integer.
Then we have $\det(T_{4,x}) = x$.
\end{lemma}
\begin{proof}
Set $f_{a,b}(t) = (t^{ab}-1)(t-1)$ and $g_{a,b}(t) = (t^a-1)(t^b-1)$. 
The Alexander polynomial of a torus knot $T_{a,b}$ is given by $\Delta_{T_{a,b}}(t)=f_{a,b}(t)/g_{a,b}(t)$. 
See~\cite{KawauchiBook} or \cite{RolfsenBook} for example. 
For a positive odd integer $x$, we see that 
$$f_{4,x}(-1) = ((-1)^{4x}-1)((-1)-1) =0,$$
$$g_{4,x}(-1) = ((-1)^{4}-1)((-1)^x-1) =0,$$
\begin{align*}
\left.\frac{d}{dt}f_{4,x}(t)\right|_{t=-1} &= \left(4xt^{4x-1}(t-1) + (t^{4x}-1)\right)|_{t=-1}\\ 
&= \left(4x(-1)^{4x-1}(-1-1) + ((-1)^{4x}-1)\right)\\ 
&= 8x, 
\end{align*}
and
\begin{align*}
\left.\frac{d}{dt}g_{4,x}(t)\right|_{t=-1} &= \left(xt^{x-1}(t^4-1) + 4t^3(t^{x}-1)\right)|_{t=-1}\\ 
&= \left(x(-1)^{x-1}((-1)^4-1) + 4(-1)^3((-1)^{x}-1)\right)\\ 
&= 8. 
\end{align*}
Then, by L'H\^opital's rule, we have 
$\det (T_{4,x}) = 8x/8 =x$. 
\end{proof}

\subsection{Criterion for a knot to be a Montesinos knot}\label{subsec:M-link}

In general, it is difficult to detect whether a given diagram represents a Montesinos knot. 
Here we give a new criterion to this problem, which is obtained by considering the alternation number of a knot. 
For two links $L$ and $L'$, the {\it Gordian distance} $d^\mathrm{x}(L,L')$ between $L$ and $L'$ is defined to be the minimal number of the crossing changes needed to deform $L$ into $L'$. 
For details and studies on the Gordian distance 
we refer the reader to \cite{Baader2005a}, \cite{HirasawaUchida2002}, \cite{IchiharaJong2010}, \cite{Nakanishi2007, NakanishiOhyama2006}, \cite{Ohyama2006}. 
Let $\mathcal{A}$ be the set of all alternating links containing trivial ones. 
The {\it alternation number} of a link $L$~\cite{Kawauchi2010}, denoted by $\alt(L)$, is defined by
$$\alt(L) = \min_{L' \in \mathcal{A}}d^\mathrm{x}(L,L'). $$
In other words, the alternation number of $L$ is the minimal number of crossing changes needed to deform $L$ into an alternating link. 
Note that if $L$ is alternating, then $\alt(L)=0$. 

For a knot $K$, we denote by $\sigma(K)$ the {\it signature} of $K$ and by $s(K)$ the {\it Rasmussen invariant} of $K$, which are introduced by Murasugi~\cite{Murasugi1965} and Rasmussen~\cite{Rasmussen2004} respectively. 
We fix the sign convention of them so that $$s(\text{right-handed trefoil knot}) = 2 \text{ and } \sigma(\text{right-handed trefoil knot}) = -2.$$ 
It is known that these two invariants have similar properties, 
in particular, $\sigma(K)$ and $s(K)$ are even integers for any knot $K$. 
In addition, it is known that how a crossing change affects on $\sigma(K)$ and $s(K)$ as follows. 
Let $K_+$ and $K_-$ be knots with diagrams $D_+$ and $D_-$ which differ only in a small neighborhood as shown in Figure~\ref{fig:+-}. 
Then we have the following. 

\begin{lemma}[see for example {\cite[Theorem 6.4.7]{MurasugiBook}}]\label{lem:+-cc_sig}
$\sigma(K_+) \le \sigma(K_-) \le \sigma(K_+) + 2.$
\end{lemma}
\begin{lemma}[{\cite[Corollary 4.3]{Rasmussen2004}}]\label{lem:+-cc_s}
$s(K_-) \le s(K_+) \le s(K_-) + 2.$
\end{lemma}
\begin{figure}[htb]
\includegraphics[width=.2\textwidth]{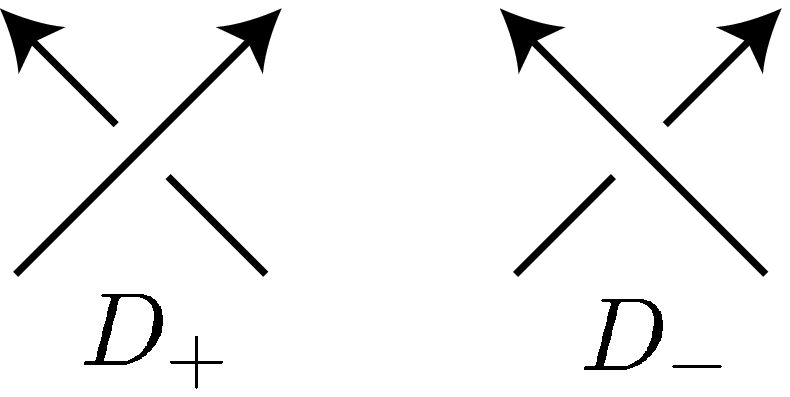}
\caption{}\label{fig:+-}
\end{figure}

Then we have the following. 

\begin{criterion}\label{cri:Monte}
For a knot $K$, if $|s(K) + \sigma(K)| \ge 4$, then $K$ is not equivalent to a Montesinos knot. 
\end{criterion}
\begin{proof}
For a knot $K$, we have $|s(K) + \sigma(K)|/2 \leq \alt(K)$~\cite[Corollary 1.3]{Abe2009}. 
Here we note that if $K$ is alternating, then we have $s(K) + \sigma(K) =0$~\cite[Theorem 3]{Rasmussen2004}. 
It was shown by Abe, Kishimoto, and the second author that the alternation number of a Montesinos link is zero or one~\cite[Proposition A.5]{AbeKishimoto2009}. 
This completes the proof. 
\end{proof}

\subsection{Lemmas to estimate the Rasmussen invariant and the signature}\label{subsec:lemmas} 

To use Criterion~\ref{cri:Monte}, we need to calculate or estimate the Rasmussen invariant and the signature of a knot. 
First we give a lemma to estimate the Rasmussen invariant. 
For a knot $K$, we denote by $g(K)$ the genus (i.e.~the minimal genus of Seifert surfaces for $K$) and by $g_*(K)$ the slice genus (i.e.~the minimal genus of a smoothly embedded orientable surface in the $4$-ball, whose boundary is $K$). 
For a diagram $D$, we denote by $c(D)$ the number of crossings of $D$, 
by $O(D)$ the number of Seifert circles of $D$, and by $w(D)$ the writhe of $D$, that is, the number of positive crossings of $D$ minus the number of negative crossings of $D$. 
A diagram $D$ is called {\it positive} (resp.~{\it negative}) if all crossings of $D$ are positive (resp.~negative), and a knot is called {\it positive} if it admits a positive diagram. 

\begin{lemma}\label{lem:sharp-rasmussen}
Let $K$ and $K'$ be positive knots with positive diagrams $D$ and $D'$ 
which differ only in a small neighborhood as shown in Figure~\ref{fig:sharp_2}. 
Then we have $$s(K) - s(K') = 8.$$ 
\end{lemma}
\begin{proof}
Let $K$ and $K'$ be positive knots with positive diagrams $D$ and $D'$ 
which differ only in a small neighborhood as shown in Figure~\ref{fig:sharp_2}. 
Then we have $c(D) = c(D') + 8$ and $O(D) = O(D')$. 
For a positive knot $K$ with a positive diagram $D$, 
we have $s(K) = 2g_*(K) =2g(K)$~\cite[Theorem 4]{Rasmussen2004}. 
Furthermore we have $2g(K) = c(D) -O(D) +1$ since a Seifert surface obtained from a positive diagram by applying Seifert's algorithm is minimal genus~\cite[Corollary 4.1]{Cromwell1989}. 
Thus we have 
$s(K) = c(D) - O(D) +1$ and $s(K') = (c(D') + 8) + O(D) +1$, and then we have $s(K) - s(K') = 8$. 
\begin{figure}[htb]
\includegraphics[width=.35\textwidth]{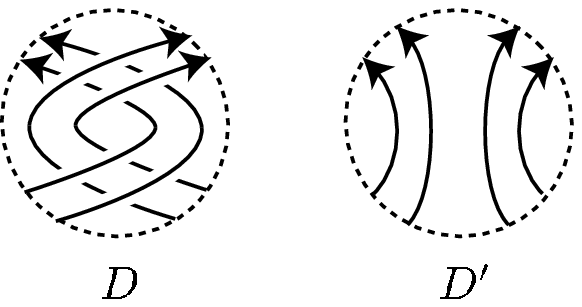}
\caption{}\label{fig:sharp_2}
\end{figure}
\end{proof}

\begin{remark}\label{rem:positiveS-sigma}
It is known that $|\sigma(K)| \le 2g_*(K)$ holds for a knot $K$. 
Therefore we have $s(K) + \sigma(K) \ge 0$ if $K$ is positive knot. 
\end{remark}

Next we give a lemma to estimate the signature of a knot. 

\begin{lemma}[{\cite[Theorem 3.2]{Murakami1985}}]\label{lem:sharp-sign}
Let $K$ and $K'$ be knots with diagrams $D$ and $D'$ which differ only in a small neighborhood as shown in Figure~\ref{fig:sharp}. 
Let $D_0$ be the diagram which differ from $D$ (and $D'$) in a small neighborhood as shown in Figure~\ref{fig:sharp0}. 
If $D_0$ is a 2-component link diagram, then we have $$2 \le \sigma(K') -\sigma(K) \le 4.$$ 
\end{lemma}
The proof of Lemma~\ref{lem:sharp-sign} is achieved by direct calculations of the signatures by using the method due to Gordon and Litherland~\cite{GordonLitherland1978}. 
For the precise proof, see \cite[Case I of the proof of Theorem 3.2]{Murakami1985}. 
We note that in Lemma~\ref{lem:sharp-sign}, if $D_0$ is a knot diagram, then we have $ 2 \le \sigma(K') -\sigma(K) \le 6$. 

\begin{figure}[htb]
\includegraphics[width=.4\textwidth]{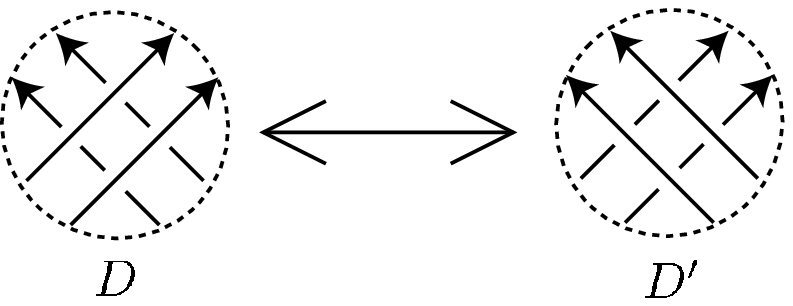}
\caption{A \#-move}\label{fig:sharp}
\end{figure}
\begin{figure}[htb]
\includegraphics[width=.1\textwidth]{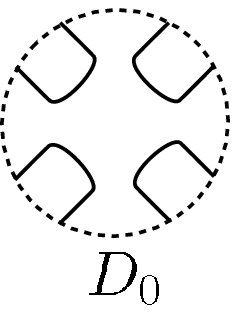}
\caption{}\label{fig:sharp0}
\end{figure}

\section{Proof of Theorem~\ref{thm:inv}}\label{sec:proof}

We divide the proof of Theorem~\ref{thm:inv} into two propositions (Propositions~\ref{prop:odd} and \ref{prop:even}) 
based on the parity of the integer $p$ 
in the statement of the theorem. 
The two propositions will be proved 
in the following two subsections respectively.

\subsection{The case where $p$ is odd}

In this subsection, we prove the following. 

\begin{proposition}\label{prop:odd}
Let $K=P(p,q,q)$ for odd integers $p, q \geq 3$. 
Then $K$ admits no Seifert fibered surgery.
\end{proposition}
\begin{proof}
Set $K=P(p,q,q)$ for odd integers $p,q \ge 3$. 
By the results in \cite{FuterIshikawaKabayaMattmanShimokawa2009}, \cite{IchiharaJong2009}, and \cite{IchiharaJong2010a}, 
if $K(r)$ is a Seifert fibered manifold for $r \in \mathbb{Q}$, then $K(r)$ is atoroidal and has the infinite fundamental group, 
in particular, the base orbifold must be $S^2$ having just three exceptional fibers. 
This fact is also implied by \cite[Corollary 1.4]{Motegi2003} since $K$ admits two different strong inversions. 
Thus it suffices to show that $K(r)$ is not a Seifert fibered manifold with the base orbifold $S^2$ having just three exceptional fibers for any $r \in \mathbb{Q}$. 

The proof is achieved by two steps as follows: 
First, we find a restriction on $r$, and next we use Proposition~\ref{cri:SFSquotient} to complete the proof.

\begin{claim}\label{clm:slope-odd}
If $K(r)$ is a Seifert fibered manifold, then $r$ is an integer with $|r| \leq 8$. 
\end{claim}
\begin{proof}
Suppose that $K(r)$ is a Seifert fibered manifold for $r \in \mathbb{Q}$. 
Then $r$ is an integer since $K$ is alternating~\cite[Theorem 1.1]{Ichihara2008a}. 
Since $K$ is a knot of genus one, the exterior of $K$ contains an essential once-punctured torus as a Seifert surface for $K$. 
Thus, the manifold $K(0)$ contains an essential torus, and the $0$-surgery is exceptional. 
Therefore we have $|r| \leq 8$~\cite[Theorem 1.1]{Ichihara2008} (see also \cite[Theorem 1.2]{LackenbyMeyerhoff2008}). 
\end{proof}

It suffices to show that $K(r)$ is not a Seifert fibered manifold with the base orbifold $S^2$ having just three exceptional fibers for $r \in \mathbb{Z}$ with $|r| \leq 8$. 
Next we apply Proposition~\ref{cri:SFSquotient} to complete the proof of Proposition~\ref{prop:odd}. 
Take an axis $\alpha_1$ as shown in Figure~\ref{fig:axis1} (left), which induces a strong inversion of $K$. 
Applying the Montesinos trick, we obtain the link $K_o = K_o(p,q,r)$ which is represented by the diagram $D_o(p,q,r)$ as shown in Figure~\ref{fig:axis1} (center and right). 
By Proposition~\ref{cri:SFSquotient}, it suffices to show that $K_o$ is equivalent to neither a Montesinos link nor a Seifert link. 
Note that $K_o$ is either a knot if $r$ is odd, or a $2$-component link if $r$ is even (see Figure~\ref{fig:axis1}). 
In the following two claims (Claims~\ref{clm:odd-link1}, \ref{clm:odd-link2}), we consider the case where $K_o$ is a $2$-component link, namely the case where $r$ is even.

\begin{figure}[htb]
\includegraphics[width=.9\textwidth]{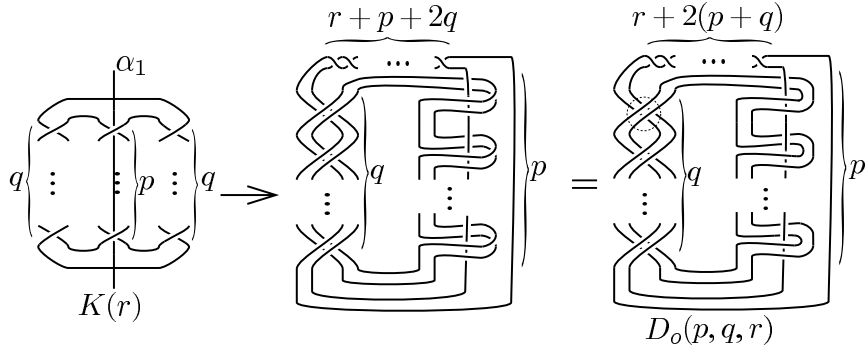}
\caption{The link $K_o = K_o(p,q,r)$ is represented by the diagram $D_o(p,q,r)$.}\label{fig:axis1}
\end{figure}

\begin{claim}\label{clm:odd-link1}
If $r$ is an even integer, then $K_o$ is not equivalent to a Montesinos link. 
\end{claim}
\begin{proof}
Assume that $r$ is an even integer. 
Suppose for a contradiction that $K_o$ is equivalent to a Montesinos link. 
We may assume that the length of $K_o$ equals three since $K(r)$ has just three exceptional fibers if $K(r)$ is Seifert fibered. 
Then we have $b(K_o) \leq 3$ (see \cite{BoileauZieschang1985} for example), where $b(L)$ denotes the bridge index of a link $L$. 
On the other hand,  $K_o$ is a $2$-component link which consists of $T_{2,q}$ and $T_{2,2p+q}$ (see Figure~\ref{fig:axis1}). 
Since each component of $K_o$ is non-trivial, we have $b(K_o) \geq 4$. 
Now we have a contradiction and complete the proof of Claim~\ref{clm:odd-link1}. 
\end{proof}

\begin{remark}
In the proof of Claim~\ref{clm:odd-link1}, we use the bridge index to show that a given link is non-Montesinos 
since Criterion~\ref{cri:Monte} works only for knots. 
\end{remark}

\begin{claim}\label{clm:odd-link2}
If $r$ is an even integer, then $K_o$ is not equivalent to a Seifert link. 
\end{claim}
\begin{proof}
Assume that $r$ is even. 
Then, as mentioned above, $K_o$ is a $2$-component link which consists of $T_{2,q}$ and $T_{2,2p+q}$. 
If $K_o$ is a Seifert link, then $K_o$ is equivalent to a link of type (S2) or (S3) in Lemma~\ref{lem:SeifertLink}. 
However $K_o$ is not appropriate for them since a link of type (S2) consists of parallel components and a link of type (S3) possesses a trivial component. 
\end{proof}

Now we may assume that $r$ is odd with $|r| \le 7$, namely, $K_o$ is a knot.

\begin{claim}\label{clm:Monte-odd}
$K_o$ is not equivalent to a Montesinos knot. 
\end{claim}
\begin{proof}
Since $r+2(p+q) \geq -7 + 2(3+3) = 5 > 0$, the diagram $D_o(p,q,r)$ is a positive diagram, and then $K_o$ is a positive knot (see Figure~\ref{fig:axis1}). 
Let $K_o'$ be the knot obtained from $K_o$ by applying a \#-move in the broken circle depicted in Figure~\ref{fig:axis1}. 
Here a \#-move is a local operation on diagrams as shown in Figure~\ref{fig:sharp}. 
Note that $K_o'$ is also positive since $q \geq 3$. 
Let $D_0$ be the diagram obtained from $D_o(p,q,r)$ 
by changing the tangle diagram in the broken circle depicted in Figure~\ref{fig:axis1} 
to that depicted in Figure~\ref{fig:sharp0}. 
Notice that $D_0$ is a 2-component link diagram. 
Then by Lemma~\ref{lem:sharp-sign}, we have 
$$\sigma(K_o') - \sigma(K_o) \leq 4. $$ 
On the other hand, by Lemma~\ref{lem:sharp-rasmussen}, we have 
$$s(K_o) - s(K_o') = 8. $$
By Remark~\ref{rem:positiveS-sigma}, we have 
$$ s(K_o') + \sigma(K_o') \ge 0. $$ 
Thus, we have
\begin{align*}
s(K_o) + \sigma(K_o) & \ge s(K_o') + 8 + \left(\sigma(K_o')-  4 \right) \\
& = s(K_o') + \sigma(K_o') + 4 \\
& \geq 4.
\end{align*}
Therefore $K_o$ is not equivalent to a Montesinos knot by Criterion~\ref{cri:Monte}. 
\end{proof}

\begin{remark}\label{rem:ConstNonMonte}
The argument in the proof of Claim~\ref{clm:Monte-odd} possibly gives a method to construct examples of non-Montesinos knots. 
\end{remark}

We denote by $\sigma_i$ ($i=1,\dots,n-1$) the Artin's standard generator of the $n$ strands braid group $B_n$ and by $\Delta_n^2$ ($\in B_n$) the braid represented by a positive full twist on $n$ strands. 
For basic terminologies in braid theory, we refer the reader to \cite{MurasugiBook}, \cite{MurasugiKurpitaBook}.

\begin{claim}\label{clm:4torus-odd}
$K_o$ is not equivalent to a torus knot. 
\end{claim}
\begin{proof}
As shown in Figure~\ref{fig:axis1}, $K_o$ is the closure of a four-strand positive braid
$$(\sigma_2\sigma_3\sigma_1\sigma_2)^{q} (\sigma_2\sigma_3^2\sigma_2)^p \sigma_1^{2p+2q+r}.$$ 
Then we can easily see that this braid is equivalent to a four-strand positive braid with a full twist as follows: 
\begin{align*}
(\sigma_2\sigma_3\sigma_1\sigma_2)^{q} (\sigma_2\sigma_3^2\sigma_2)^p \sigma_1^{2p+2q+r} 
&= \sigma_1^4 (\sigma_2\sigma_3\sigma_1\sigma_2)^{q} (\sigma_2\sigma_3^2\sigma_2)^p \sigma_1^{2p+2q+r-4} \\
&= (\sigma_2\sigma_3\sigma_1\sigma_2) \sigma_3^4 (\sigma_2\sigma_3\sigma_1\sigma_2)^{q-1} (\sigma_2\sigma_3^2\sigma_2)^p \sigma_1^{2p+2q+r-4} \\
&= (\sigma_2\sigma_3\sigma_1\sigma_2) \sigma_3^2 (\sigma_2\sigma_3\sigma_1\sigma_2)^{q-1} \sigma_3^2 (\sigma_2\sigma_3^2\sigma_2)^p \sigma_1^{2p+2q+r-4} \\
&= \sigma_3^2 (\sigma_2\sigma_3\sigma_1\sigma_2) \sigma_3^2 (\sigma_2\sigma_3\sigma_1\sigma_2)^{q-1} (\sigma_2\sigma_3^2\sigma_2)^p \sigma_1^{2p+2q+r-4} \\
&= (\sigma_2\sigma_3\sigma_1\sigma_2) \sigma_1^2 \sigma_3^2 (\sigma_2\sigma_3\sigma_1\sigma_2)^{q-1} (\sigma_2\sigma_3^2\sigma_2)^p \sigma_1^{2p+2q+r-4} \\
&= \Delta_4^2(\sigma_2\sigma_3\sigma_1\sigma_2)^{q-2} (\sigma_2\sigma_3^2\sigma_2)^p \sigma_1^{2p+2q+r-4} 
\end{align*}
(see Figure~\ref{fig:fulltwist}). 
Note that $q -2 \ge 3-2 = 1$ and $2p + 2q +r -4 \ge 2(3+3) -7 -4 =1$. 
This implies that the braid index of $K_o$ equals four~\cite[Corollary 2.4]{FranksWilliams1987}. 
Thus, it suffices to show that $K_o$ is not equivalent to a torus knot $T_{4,x}$. 
Suppose for a contradiction that $K_o$ is equivalent to a torus knot $T_{4,x}$. 
We may assume that $x$ is positive since $K_o$ is a positive knot. 
Since $K(r)$ is the double branched covering space of $S^3$ branched along $K_o$, we have 
$$\det(K_o) = |H_1(K(r))| $$ 
(see \cite[p.~213]{RolfsenBook} for example). 
Since $r$ is an integer, we have 
$$|H_1(K(r))| = |r|. $$ 
On the other hand, by Lemma~\ref{lem:detT4}, we have 
$$\det(T_{4,x})=x. $$
Thus, we have $$x=|r|.$$
Here we note that $-7 \leq r \leq 7$, and thus $x \le 7$. 

Next we calculate the genera of $K_o$ and $T_{4,x}$. 
Since $K_o$ is a positive knot with a positive diagram $D_o(p,q,r)$, 
$$g(K_o) = (c(D_o(p,q,r)) - O(D_o(p,q,r)) + 1)/2$$ 
holds as claimed in the proof of Lemma~\ref{lem:sharp-rasmussen}. 
Then we have 
\begin{align*}
g(K_o) &= (\left(4p+4q+2p+2q+r\right)-4+1)/2\\
&= 3(p+q) + (r-3)/2\\
&\geq 3(3+3) + (-7-3)/2 = 13.
\end{align*}
On the other hand, since $T_{4,x}$ is also positive, using the standard diagram of a torus knot, we have
\begin{align*}
g(T_{4,x}) &= 3(x-1)/2\\ 
&\le 3(7-1)/2 = 9. 
\end{align*}
Thus, we have a contradiction and complete the proof of Claim~\ref{clm:4torus-odd}. 
\end{proof}
Now we complete the proof of Proposition~\ref{prop:odd}. 
\end{proof}

\begin{figure}[htb]
\includegraphics[width=.2\textwidth]{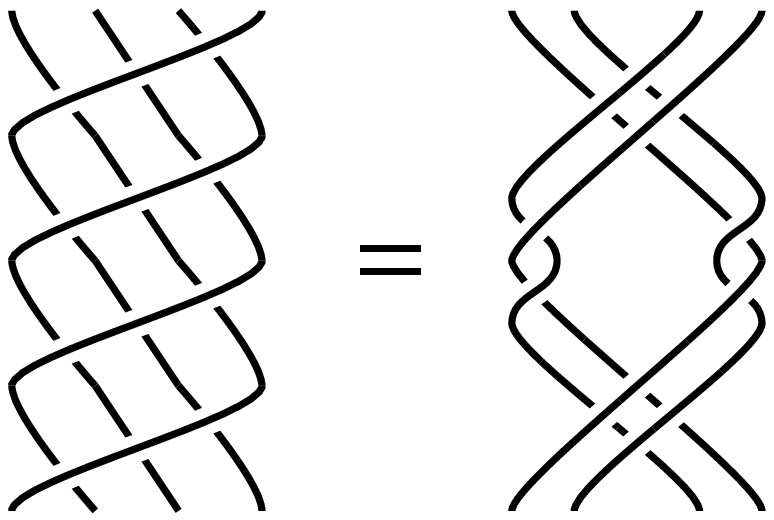}
\caption{}\label{fig:fulltwist}
\end{figure}

\begin{remark}
We can use the Rasmussen invariant instead of the genus in the last step of the proof of Claim~\ref{clm:4torus-odd} since $s(K) = 2g(K) = c(D) - O(D) +1$ for a positive knot $K$ with a positive diagram $D$. 
\end{remark}

\subsection{The case where $p$ is even}

In this subsection, we prove the following. 

\begin{proposition}\label{prop:even}
Let $K=P(2n,q,q)$ for an integer $n \geq 1$ and an odd integer $q \geq 3$. 
Then $K$ admits no Seifert fibered surgery.
\end{proposition}

The outline of the proof is similar to that of Proposition~\ref{prop:odd}. 
Before we start the proof of Proposition~\ref{prop:even}, 
we consider Seifert fibered surgeries on periodic knots. 
We say a knot $K$ in $S^3$ has a {\it cyclic period} $p$ if there exists an orientation preserving homeomorphism $f : S^3 \rightarrow S^3$ such that $f(K) = K$, $f^p=\id$ ($p>1$), $\fix(f) \neq \emptyset$, and $\fix(f) \cap K = \emptyset$. 
By the affirmative answer to the Smith Conjecture~\cite{MorganBass1984}, the map $f$ is a rotation of $S^3$ about the unknotted circle $\fix(f)$. 
So by taking the quotient $S^3/f$, we obtain $S^3 = S^3/f$ and a new knot $K_f=K/f$. 
We call $K_f$ the \textit{factor knot} of $K$ with respect to the cyclic period. 

By virtue of the studies on Seifert fibered surgeries on periodic knots due to Miyazaki and Motegi~\cite{MiyazakiMotegi2002a}, and also Motegi~\cite{Motegi2003}, we obtain the following.


\begin{lemma}[\cite{MiyazakiMotegi2002a}]\label{lem:period}
Let $K$ be a hyperbolic knot which has a cyclic period $2$, 
and $K'$ the factor knot of $K$. 
Suppose that $K(r)$ is a Seifert fibered manifold with the base orbifold $S^2$ having just three exceptional fibers for $r \in \mathbb{Z}$. 
If $K'$ is equivalent to a torus knot $T_{2,q}$ with $q \ge 3$, 
then we have $r= 4q \pm 1$. 
\end{lemma}
\begin{proof}
We only show the outline of the proof. 
For details, see \cite[Section 2]{MiyazakiMotegi2002a}. 
Let $f$ be a periodic map inducing a cyclic period $2$ on $K$. 
Suppose that $K(r)$ is a Seifert fibered manifold with the base orbifold $S^2$ having just three exceptional fibers, 
and let $\pi: K(r) \rightarrow S^2 
$ be a Seifert fibration. 
Let $\bar{f}: K(r) \rightarrow K(r)$ be the natural extension of $f$, 
which preserves the Seifert fibration 
(actually, we can take $\bar{f}$ as fiber preserving~\cite[Lemma 2.1]{MiyazakiMotegi2002a}), 
and let $\hat{f} : S^2
\rightarrow S^2 
$ be a homeomorphism satisfying $\hat{f}\circ \pi = \pi \circ \bar{f}$. 
Then $\hat{f}$ is a reflection along the equator 
$C_f =\fix(\hat{f}) = \pi(\fix(\bar{f})) \subset S^2 
$~\cite[Lemma 2.3 (1)]{MiyazakiMotegi2002a}. 
Here we note that the following diagram commutes: 
$$\xymatrix{
S^3 \ar[d]_{r\text{-surgery}} \ar[r]^{/f} \ar@{}[dr]|\circlearrowleft & S^3/f=S^3 \ar@{>}[dr]^{r/2\text{-surgery}}&~ \\
K(r) \ar[r]^{/\bar{f}} & K(r)/\bar{f} & \hspace{-35pt} \cong K'(r/2)
}$$
Since the base orbifold of $K(r)$ is $S^2$ and $K(r)$ has just three exceptional fibers, the configuration of cone points is either of the following: 
\begin{enumerate}
\item[(i)] Three cone points are lying on $C_f$. 
\item[(ii)] One is lying on $C_f$, another one is in the north hemisphere, and the other one is in the south hemisphere respectively. 
\end{enumerate}
Thus $K(r)/\bar{f}$ is homeomorphic to either $S^3$ if the configuration is of type (i) or a lens space if the configuration is of type (ii). 
If $K'$ is a torus knot $T_{2,q}$ with $q \ge 3$, then $K'(r/2)$ is not homeomorphic to $S^3$~\cite{Moser1971} and thus (i) is not appropriate. 
Therefore two cone points in each hemispheres have the same index and $K(r)/\bar{f} \cong K'(r/2)$ is equivalent to a lens space. 
Then by the classification of Dehn surgeries on torus knots due to Moser~\cite{Moser1971}, we have $r/2 = 2q \pm 1/2$. 
Since $r \in \mathbb{Z}$, we have $r=4q \pm 1$. 
\end{proof}


Now we start the proof of Proposition~\ref{prop:even}. 

\begin{proof}[Proof of Proposition~\ref{prop:even}]
Set $K=P(2n,q,q)$. 
By the same reason as in the proof of Proposition~\ref{prop:odd}, 
it suffices to show that $K(r)$ is not a Seifert fibered manifold with the base orbifold $S^2$ having just three exceptional fibers for any $r \in \mathbb{Q}$. 

To restrict Seifert fibered slopes, we apply Lemma~\ref{lem:period}. 

\begin{claim}
If $K(r)$ is a Seifert fibered manifold, then we have $r=4q \pm 1$. 
\end{claim}
\begin{proof}
Suppose that $K(r)$ is a Seifert fibered manifold with the base orbifold $S^2$ having just three exceptional fibers. 
As shown in Figure~\ref{fig:cyclicperi1}, $K$ has a cyclic period $2$, 
and the factor knot with respect to this cyclic period is equivalent to $T_{2,q}$. 
Then by Lemma~\ref{lem:period}, we have 
$r/2 = 2q \pm 1/2$, that is, $r=4q \pm 1$. 
\begin{figure}[htb]
\includegraphics[width=.4\textwidth]{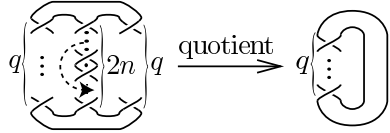}
\caption{The $\pi$-rotation (with axis perpendicular to the page) induces a cyclic period $2$.}\label{fig:cyclicperi1}
\end{figure}
\end{proof}

Next we use Proposition~\ref{cri:SFSquotient} to complete the proof of Proposition~\ref{prop:even}. 
We assume that $r=4q \pm 1$. 
Take an axis $\alpha_2$ which induces a strong inversion of $K$ as shown in Figure~\ref{fig:axis2} (left). 
Applying the Montesinos trick, we obtain the knot $K_e = K_e(n,q,r)$ which is represented by the diagram $D_e(n,q,r)$ as shown in Figure~\ref{fig:axis2} (center and right). 
Here we note that $K_e$ is a knot since $r=4q \pm1$ is an odd integer. 
Then by Proposition~\ref{cri:SFSquotient}, it suffices to show that $K_e$ is equivalent to neither a Montesinos knot nor a torus knot. 

\begin{figure}[htb]
\includegraphics[width=.9\textwidth]{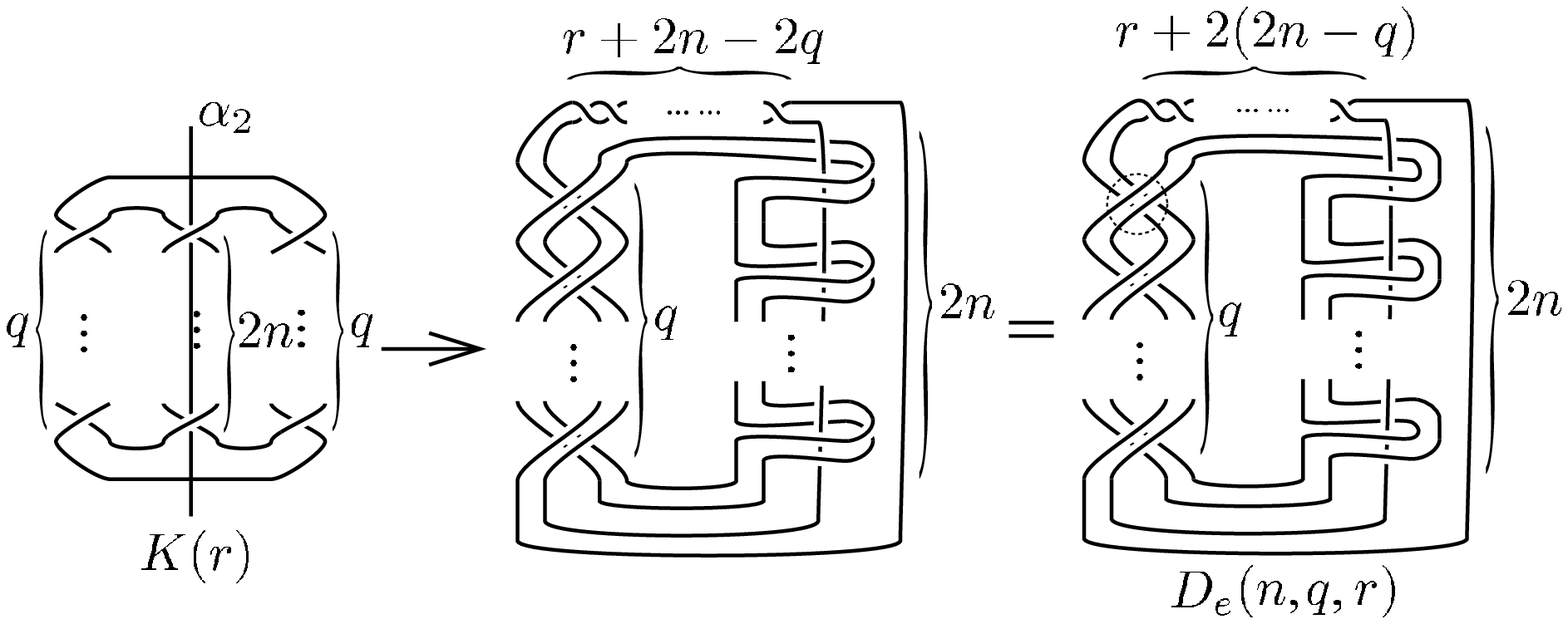}
\caption{}\label{fig:axis2}
\end{figure}

\begin{claim}\label{clm:Monte-even}
$K_e$ is not equivalent to a Montesinos knot. 
\end{claim}
\begin{proof}
The diagram $D_e(n,q,r)$ is a positive diagram since 
\begin{align*}
r + 2(2n-q) &\ge 4q -1 + 2(2n-q) \\ 
&\ge 12 -1 + 2(2-3) \\ 
&= 9 > 0. 
\end{align*}
Thus, $K_e$ is a positive knot. 
Let $K_e'$ be the knot obtained from $K_e$ by applying a \#-move in the broken circle depicted in Figure~\ref{fig:axis2}.  
Note that $K_e'$ is also positive since $q \geq 3$. 
Let $D_0$ be the diagram obtained from $D_e(p,q,r)$ 
by changing the tangle diagram in the broken circle depicted in Figure~\ref{fig:axis2} to that depicted in Figure~\ref{fig:sharp0}. 
Notice that $D_0$ is a 2-component link diagram. 
Then, by Lemma~\ref{lem:sharp-sign}, we have $$\sigma(K_e') - \sigma(K_e) \leq 4.$$ 
On the other hand, by Lemma~\ref{lem:sharp-rasmussen}, we have $$s(K_e) - s(K_e') = 8.$$ 
By Remark~\ref{rem:positiveS-sigma}, we have $$s(K_e') + \sigma(K_e') \ge 0. $$
Thus, we have
\begin{align*}
s(K_e) + \sigma(K_e) & \ge s(K_e') + 8 + \left(\sigma(K_e') - 4 \right) \\
& = s(K_e') + \sigma(K_e') + 4 \\
& \geq 4.
\end{align*}
Therefore $K_e$ is not equivalent to a Montesinos knot by Criterion~\ref{cri:Monte}. 
\end{proof}

\begin{claim}\label{clm:4torus-even}
$K_e$ is not equivalent to a torus knot. 
\end{claim}
\begin{proof}
As shown in Figure~\ref{fig:axis2}, $K_e$ is the closure of a four-strand positive braid 
$$(\sigma_2\sigma_3\sigma_1\sigma_2)^{q} (\sigma_2\sigma_3^2\sigma_2)^p \sigma_1^{2(2n-q)+r} = \Delta_4^2 (\sigma_2\sigma_3\sigma_1\sigma_2)^{q-2} (\sigma_2\sigma_3^2\sigma_2)^p \sigma_1^{2(2n-q)+r-4}.$$
Then the braid index of $K_e$ equals four~\cite[Corollary 2.4]{FranksWilliams1987}. 
Thus, it suffices to show that $K_e$ is not equivalent to a torus knot $T_{4,x}$. 
Suppose for a contradiction that $K_e$ is equivalent to a torus knot $T_{4,x}$. 
We may assume that $x$ is positive since $K_e(p,q,r)$ is a positive knot. 
Since $K(r)$ is the double branched covering space of $S^3$ branched along $K_e$, we have
$$\det(K_e) = |H_1(K(r))|. $$
Since $r=4q \pm 1$ is a positive integer, we have 
$$|H_1(K(r))| = r =4q \pm1.$$ 
On the other hand, by Lemma~\ref{lem:detT4}, we have 
$$\det(T_{4,x})=x. $$ 
Thus, we have $x=r = 4q \pm 1$. 
Now we calculate the genera of $K_e(n,q,4q \pm 1)$ and of $T_{4,4q \pm 1}$ in the same way as in the proof of Claim~\ref{clm:4torus-odd}. 
We have 
\begin{align*}
g(K_e(n,q,4q \pm 1)) &= (4q + 8n + 4n -2q +4q \pm 1 -4 + 1)/2\\
&=
\begin{cases}
6n+3q-1 & \textrm{ if \ } r=4q+1,\\
6n+3q-2 & \textrm{ if \ } r=4q-1,
\end{cases}
\end{align*}
and
\begin{align*}
g(T_{4,r}) = 3(r-1)/2
&=\begin{cases}
6q &  \textrm{ if \ } r=4q+1,\\
6q -3 &  \textrm{ if \ } r=4q-1.
\end{cases} 
\end{align*}
Therefore we have 
\begin{align*}
\begin{cases}
6n+3q-1 = 6q & \textrm{ if \ } r=4q+1,\\
6n+3q-2 = 6q-3 & \textrm{ if \ } r=4q-1.
\end{cases}
\end{align*}
That is, we have 
\begin{align*}
\begin{cases}
6n-3q = 1 & \textrm{ if \ } r=4q+1,\\
6n-3q = -1 & \textrm{ if \ } r=4q-1.
\end{cases}
\end{align*}
These contradict that $n$ and $q$ are integers, and we complete the proof of Claim~\ref{clm:4torus-even}. 
\end{proof}
Now we complete the proof of Proposition~\ref{prop:even}. 
\end{proof}

\subsection*{Acknowledgments}
The authors would like to thank Professor Taizo Kanenobu 
for useful comments on Proposition~\ref{cri:SFSquotient}. 
They also  wish to thank Professor Akio Kawauchi for helpful conversations. 
They also would like to thank Tetsuya Ito for letting them know the result obtained by Franks and Williams in~\cite{FranksWilliams1987}. 
They also would like to thank the referee for careful reading and useful suggestions.

\providecommand{\bysame}{\leavevmode\hbox to3em{\hrulefill}\thinspace}
\providecommand{\MR}{\relax\ifhmode\unskip\space\fi MR }
\providecommand{\MRhref}[2]{%
  \href{http://www.ams.org/mathscinet-getitem?mr=#1}{#2}
}
\providecommand{\href}[2]{#2}


\begin{thebibliography}{10}

\bibitem{Abe2009}
T.~Abe, \emph{An estimation of the alternation number of a torus knot}, J. Knot
  Theory Ramifications \textbf{18} (2009), 363--379.

\bibitem{AbeKishimoto2009}
T.~Abe and K.~Kishimoto, \emph{{The dealternating number and the alternation
  number of a closed $3$-braid}}, J. Knot Theory Ramifications \textbf{19}
  (2010), no.~9, 1157--1181.

\bibitem{Baader2005a}
S.~Baader, \emph{{Slice and Gordian numbers of track knots}}, Osaka J. Math.
  \textbf{42} (2005), no.~1, 257--271.

\bibitem{BoileauZieschang1985}
M.~Boileau and H.~Zieschang, \emph{{Nombre de ponts et g\'en\'erateurs
  m\'eridiens des entrelacs de Montesinos}}, Comment. Math. Helv. \textbf{60}
  (1985), no.~2, 270--279.

\bibitem{BonahonSiebenmann1979-85}
F.~Bonahon and L.~C.~Siebenmann, \emph{{Geometric splittings of knots, and
  Conway's algebraic knots}}, Draft of a monograph, 1979--85.

\bibitem{BonahonSiebenmann2010}
\bysame, \emph{New geometric splittings of classical knots and the
  classification and symmetries of arborescent knots}, preprint (2010).

\bibitem{Boyer2002}
S.~Boyer, \emph{Dehn surgery on knots}, Handbook of Geometric topology, ch.~4,
  pp.~165--218, North-Holland, Amsterdam, 2002.

\bibitem{BurdeMurasugi1970}
G.~Burde and K.~Murasugi, \emph{{Links and Seifert fiber spaces}}, Duke Math.
  J. \textbf{37} (1970), 89--93.

\bibitem{Cromwell1989}
P.~R.~Cromwell, \emph{{Homogeneous links}}, J. London Math. Soc. (2)
  \textbf{39} (1989), 535--552.

\bibitem{FuterIshikawaKabayaMattmanShimokawa2009}
D.~Futer, M.~Ishikawa, Y.~Kabaya, T.~Mattman, and K.~Shimokawa, 
\emph{Finite surgeries on three-tangle pretzel knots}, Algebr. Geom. Topol \textbf{9}
  (2009), no.~2, 743--771.

\bibitem{EisenbudNeumann1985}
D.~Eisenbud and W.~Neumann, \emph{Three-dimensional link theory and invariants
  of plane curve singularities}, Ann. of Math. Studies No.110, 1985.

\bibitem{FintushelStern1980}
R.~Fintushel and R.~J.~Stern, 
\emph{Constructing lens spaces by surgery on knots}, Math. Z. \textbf{175} (1980), no.~1, 33--51.

\bibitem{FranksWilliams1987}
J.~Franks and R.~F.~Williams, \emph{{Braids and the Jones polynomial}}, Trans.
  Amer. Math. Soc. \textbf{303} (1987), no.~1, 97--108.

\bibitem{GordonLitherland1978}
C.~McA.~Gordon and R.~A.~Litherland, 
\emph{On the signature of a link},
Invent. Math. \textbf{47} (1978), 53--69.

\bibitem{HirasawaUchida2002}
M.~Hirasawa and Y.~Uchida, \emph{{The Gordian complex of knots}}, J. Knot
  Theory Ramifications \textbf{11} (2002), no.~3, 363--368.

\bibitem{Ichihara2008a}
K.~Ichihara, \emph{All exceptional surgeries on alternating knots are integral
  surgeries}, Algebr. Geom. Topol \textbf{8} (2008), 2161--2173.

\bibitem{Ichihara2008}
\bysame, \emph{Integral non-hyperbolike surgeries}, J. Knot Theory
  Ramifications \textbf{17} (2008), no.~3, 257--261.

\bibitem{IchiharaJong2009}
K.~Ichihara and I.~D.~Jong, \emph{{Cyclic and finite surgeries on Montesinos
  knots}}, Algebr. Geom. Topol \textbf{9} (2009), no.~2, 731--742.

\bibitem{IchiharaJong2010a}
\bysame, \emph{Toroidal seifert fibered surgeries on montesinos knots}, Comm.
  Anal. Geom. \textbf{18} (2010), no.~3, 579--600.

\bibitem{IchiharaJong2010}
\bysame, \emph{{Gromov hyperbolicity and a variation of the Gordian complex}},
  Proc. Japan. Acad., Ser. A \textbf{87} (2011), no.~2, 17--21.

\bibitem{IchiharaJongKabaya2011}
K.~Ichihara, I.~D.~Jong, and Y.~Kabaya, \emph{Exceptional surgeries on $(-2,p,p)$-pretzel knots}, 
  Topology Appl. \textbf{159} (2012), no.~4, 1064--1073.

\bibitem{KawauchiBook}
A.~Kawauchi, \emph{{A Survey of knot theory}}, Birckh\"{a}user-Verlag, Basel,
  1996.

\bibitem{Kawauchi1985}
\bysame, \emph{Classification of pretzel knots}, Kobe J. Math. {\bf 2} (1985), no.~1, 11--22

\bibitem{Kawauchi2010}
\bysame, \emph{On alternation numbers of links}, Topology Appl. {\bf 157} (2010),
  no.~1, 274--279.

\bibitem{LackenbyMeyerhoff2008}
M.~Lackenby and R.~Meyerhoff, \emph{{The maximal number of exceptional Dehn
  surgeries}}, arXiv:0808.1176 (2008).

\bibitem{MiyazakiMotegi2002a}
K.~Miyazaki and K.~Motegi, \emph{{Seifert fibering surgery on periodic knots}},
 Topology Appl. \textbf{121} (2002), no.~1-2, 275--285.

\bibitem{Montesinos1975}
J.~M.~Montesinos, \emph{{Surgery on links and double branched covers of
  $S^3$}}, Knots, groups, and $3$-manifolds (Papers dedicated to the memory of
  R.~H.~Fox) Ann. of Math. Studies, No. 84 (1975), 227--259.

\bibitem{MorganBass1984}
J.~Morgan and H.~Bass, \emph{{The Smith conjecture}}, Academic Press, New York
  (1984).

\bibitem{Moser1971}
L.~Moser, \emph{Elementary surgery along a torus knot}, Pacific J. Math.
  \textbf{38} (1971), 737--745.

\bibitem{Motegi2003}
K.~Motegi, \emph{{Dehn surgeries, group actions and Seifert fiber spaces}},
  Comm. Anal. Geom. \textbf{11} (2003), no.~2, 343--389.

\bibitem{Murakami1985}
H.~Murakami, \emph{Some metrics on classical knots}, Math. Ann. \textbf{270}
  (1985), no.~1, 35--45.

\bibitem{Murasugi1965}
K.~Murasugi, \emph{{On a certain numerical invariant of link types}}, Trans.
  Amer. Math. Soc. \textbf{117} (1965), 387--422.

\bibitem{MurasugiBook}
K.~Murasugi, \emph{Knot theory and its applications}, Birckh\"{a}user-Verlag,
  Boston, 1996.

\bibitem{MurasugiKurpitaBook}
K.~Murasugi and B.~I.~Kurpita, \emph{A study of braids}, Kluwer Academic
  Publishers, 1999.

\bibitem{Nakanishi2007}
Y.~Nakanishi, \emph{{Local moves and Gordian complex II}}, Kyungpook Math. J.
  \textbf{47} (2007), 329--334.

\bibitem{NakanishiOhyama2006}
Y.~Nakanishi and Y.~Ohyama, \emph{{Local moves and Gordian complexes}}, J. Knot
  Theory Ramifications \textbf{15} (2006), no.~9, 1215--1224.

\bibitem{Oertel1984}
U.~Oertel, \emph{Closed incompressible surfaces in complements of star links},
  Pacific J. Math. \textbf{111} (1984), no.~1, 209--230.

\bibitem{Ohyama2006}
Y.~Ohyama, \emph{{The $C\sb k$-Gordian complex of knots}}, J. Knot Theory
  Ramifications \textbf{15} (2006), 73--80.

\bibitem{Rasmussen2004}
J.~A.~Rasmussen, \emph{Khovanov homology and the slice genus}, Invent. Math.
  \textbf{182} (2010), no.~2, 419--447.

\bibitem{RolfsenBook}
D.~Rolfsen, \emph{{Knots and links}}, Mathematics Lecture Series, no.~7,
  Publish or Perish, Inc., Berkeley, Calif., 1976.

\bibitem{MattmanMiyazakiMotegi2006}
{T.~W.~Mattman, K.~Miyazaki, and K.~Motegi}, \emph{{Seifert-fibered surgeries
  which do not arise from primitive/Seifert-fibered constructions}}, Trans.
  Amer. Math. Soc. \textbf{358} (2006), no.~9, 4045--4055.

\bibitem{Thurston1978}
W.~P.~Thurston, \emph{The geometry and topology of $3$-manifolds}, Lecture
  notes, Princeton University (1978), electronic version available at
  \texttt{http://www.msri.org/publications/books/gt3m}.

\bibitem{Wu1996}
Y.-Q.~Wu, \emph{{Dehn surgery on arborescent knots}}, J. Differential Geom.
  \textbf{43} (1996), no.~1, 171--197.

\bibitem{Wu2006}
\bysame, \emph{{The classification of toroidal Dehn surgeries on Montesinos
  knots}}, 
  Comm. Anal. Geom. {\bf 19} (2011), 305--345. 

\bibitem{Wu2009}
\bysame, \emph{{Immersed surfaces and Seifert fibered surgery on Montesinos
  knots}}, to appear in Trans. Amer. Math. Soc. (2011), 
    available at arXiv:0910.4882.

\bibitem{Wu2010}
\bysame, \emph{Persistently laminar branched surfaces}, 
to appear in  Comm. Anal. Geom. (2012), 
available at arXiv:1008.2680. 

\end{thebibliography}

\end{document}